\documentclass[12pt]{amsart}
\usepackage[margin=2.1cm]{geometry}
\usepackage{amssymb, amsmath, amsthm}
\usepackage{verbatim}
\usepackage{enumerate}
\usepackage{fancyhdr}
\usepackage{authblk}
\usepackage[foot]{amsaddr}

\usepackage{hyperref}

\usepackage[symbol]{footmisc}

\theoremstyle{plain}
\newtheorem{mythm}{Theorem}[section]
\newtheorem{mylemma}[mythm]{Lemma}

\newtheorem{myprop}[mythm]{Proposition}
\newtheorem{mycor}[mythm]{Corollary}

\theoremstyle{definition}
\newtheorem{mydef}[mythm]{Definition}
\newtheorem{myex}[mythm]{Example}
\newtheorem{myrmk}[mythm]{Remark}

\newcommand{\mf}[1]{\ensuremath{\mathfrak{#1}}}
\newcommand{\mb}[1]{\ensuremath{\mathbb{#1}}}
\newcommand{\mbf}[1]{\ensuremath{\mathbf{#1}}}
\newcommand{\ra}{\rightarrow}

\newcommand{\epi}{\twoheadrightarrow}
\newcommand{\mono}{\hookrightarrow}

\newcommand{\Ima}{\ensuremath{{\rm Im}}}

\newcommand{\N}{\mb{N}}
\newcommand{\Z}{\mb{Z}}
\newcommand{\Q}{\mb{Q}}
\newcommand{\R}{\mb{R}}
\newcommand{\C}{\mb{C}}
\newcommand{\F}{\mb{F}}

\newcommand{\sgn}{\ensuremath{{\rm sgn}}}
\newcommand{\End}{\ensuremath{{\rm End}}}
\newcommand{\Hom}{\ensuremath{{\rm Hom}}}
\newcommand{\Gal}{\ensuremath{{\rm Gal}}}

\newcommand{\tr}{\ensuremath{{\rm tr}}}
\newcommand{\charac}{\ensuremath{{\rm char\:}}}

\newcommand{\pl}{\ensuremath{{\rm pl}}}
\newcommand{\pf}{\ensuremath{{\rm pf}}}
\newcommand{\Map}{\ensuremath{{\rm Map}}}
\newcommand{\Spec}{\ensuremath{{\rm Spec}}}
\newcommand{\FS}{\ensuremath{{\rm FS}}}
\newcommand{\FG}{\ensuremath{{\rm FG}}}
\newcommand{\FI}{\ensuremath{{\rm FI}}}
\newcommand{\FFS}{\ensuremath{{\rm FFS}}}
\newcommand{\FFG}{\ensuremath{{\rm FFG}}}
\newcommand{\Ad}{\ensuremath{{\rm Ad}}}

\begin{document}

\title{Pseudocharacters of Classical Groups}
\author{Matthew Weidner\textsuperscript{1}\textsuperscript{*}}
\date{\today}
\address{\textsuperscript{1} Department of Mathematics, California Institute of Technology, Pasadena.}
\email{mweidner@alumni.caltech.edu}
\setcounter{Maxaffil}{0}
\renewcommand\Affilfont{\itshape\small}
\maketitle
\vspace{-10pt}

\footnotetext{\textsuperscript{*} Supported by Caltech's Samuel P.\ and Frances Krown SURF Fellowship.}

\begin{abstract}
A $GL_d$-pseudocharacter is a function from a group $\Gamma$ to a ring $k$ satisfying polynomial relations which make it ``look like'' the character of a representation.  When $k$ is an algebraically closed field, Taylor proved that $GL_d$-pseudocharacters of $\Gamma$ are the same as degree-$d$ characters of $\Gamma$ with values in $k$, hence are in bijection with equivalence classes of semisimple representations $\Gamma \ra GL_d(k)$.  Recently, V.\ Lafforgue generalized this result by showing that, for any connected reductive group $H$ over an algebraically closed field $k$ of characteristic 0 and for any group $\Gamma$, there exists an infinite collection of functions and relations which are naturally in bijection with $H^0(k)$-conjugacy classes of semisimple representations $\Gamma \ra H(k)$.  In this paper, we reformulate Lafforgue's result in terms of a new algebraic object called an $\FFG$-algebra.  We then define generating sets and generating relations for these objects and show that, for all $H$ as above, the corresponding $\FFG$-algebra is finitely presented.  Hence we can always define $H$-pseudocharacters consisting of finitely many functions satisfying finitely many relations.  Next, we use invariant theory to give explicit finite presentations of the $\FFG$-algebras for (general) orthogonal groups, (general) symplectic groups, and special orthogonal groups.  Finally, we use our pseudocharacters to answer questions about conjugacy vs.\ element-conjugacy of representations, following Larsen.
\end{abstract}

\section{Introduction}
Pseudocharacters were originally introduced for $GL_2$ by Wiles \cite{wiles} and generalized to $GL_n$ by Taylor \cite{taylor}.  Taylor's result on $GL_n$-pseudocharacters is as follows.  Let $\Gamma$ be a group and $k$ be a commutative ring with identity.  Define a \emph{$GL_n$-pseudocharacter of $\Gamma$ over $k$} to be a set map $T: \Gamma \ra k$ such that
\begin{itemize}
  \item $T(1) = n$
  \item For all $\gamma_1, \gamma_2 \in \Gamma$, $T(\gamma_1\gamma_2) = T(\gamma_2\gamma_1)$
  \item For all $\gamma_1, \dots, \gamma_{n+1} \in \Gamma$,
  \begin{equation}\label{gl_n_relation}
  \sum_{\sigma \in S_{n+1}} \sgn(\sigma) T_\sigma(\gamma_1, \dots, \gamma_{n+1}) = 0,
  \end{equation}
  where $S_{n+1}$ is the symmetric group on $n+1$ letters, $\sgn(\sigma)$ is the permutation sign of $\sigma$, and $T_\sigma$ is defined by
  \[
  T_\sigma(\gamma_1, \dots, \gamma_{n+1}) = T(\gamma_{i_1^{(1)}} \cdots \gamma_{i_{r_1}^{(1)}}) \cdots T(\gamma_{i_1^{(s)}} \cdots \gamma_{i_{r_s}^{(s)}})
  \]
  when $\sigma$ has cycle decomposition $(i_1^{(1)} \dots i_{r_1}^{(1)}) \dots (i_1^{(s)} \dots i_{r_s}^{(s)})$.
\end{itemize}
If $T$ is a $GL_n$-pseudocharacter, then define the kernel of $T$ by
\[
\ker(T) = \{\eta \in \Gamma | \mbox{$T(\gamma\eta) = T(\gamma)$ for all $\gamma \in \Gamma$}\}.
\]
Then:

\begin{mythm}[{\cite[Theorem 1]{taylor}}]
\begin{enumerate}
  \item Let $\rho: \Gamma \ra GL_n(k)$ be a representation.  Then $\tr(\rho)$ is a $GL_n$-pseudocharacter.
  \item Suppose $k$ is a field of characteristic 0, and let $\rho: \gamma \ra GL_n(k)$ be a representation.  Then $\ker(\tr(\rho)) = \ker(\rho^{ss})$, where $\rho^{ss}$ denotes the semisimplification of $\rho$.
  \item Suppose $k$ is an algebraically closed field of characteristic 0.  Let $T: \gamma \ra k$ be a $GL_n$-pseudocharacter.  Then there is a semisimple representation $\rho: \Gamma \ra GL_n(k)$ such that $\tr(\rho) = T$, unique up to conjugation.
  \item If $\Gamma$ and $k$ are taken to be topological, then the above statements hold in topological/continuous form.
\end{enumerate}
\end{mythm}

Taylor used $GL_n$-pseudocharacters to construct Galois representations having certain properties \cite[\S 2]{taylor}.

Recently, V.\ Lafforgue formulated an analog of $GL_n$-pseudocharacters which works with $GL_n$ replaced by any connected reductive group $H$.  However, instead of consisting of one function $T: \Gamma \ra k$ satisfying a finite number of relations, these ``pseudocharacters'' consist of an infinite sequence of algebra morphisms satisfying certain properties.  These sequences of morphisms are essentially equivalent to specifying an infinite number of functions $\Gamma^I \ra k$, with $I$ ranging over all finite sets, satisfying an infinite number of relations.

Lafforgue also shows how to derive Taylor's result from the above theorem \cite[Remark 11.8]{lafforgue_french}, using results of Procesi \cite{procesi} which state that the trace function ``generates'' all of the algebras $k[GL_n^I]^{\Ad GL_n}$ (here $\Ad GL_n$ denotes the diagonal conjugation action) and which explicitly describe all of the relations between these trace functions.

In Section \ref{general_results} of this paper, we reformulate Lafforgue's result in terms of a new algebraic structure called an $\FFG$-algebra.  Collections of morphisms $\Xi_n$ as above are recast as morphisms between certain $\FFG$-algebras.  We then use the finiteness theorems of classical invariant theory to show that, for any $H$ as above, the $\FFG$-algebra derived from the invariants of $H$ is finitely presented.  Hence it is always possible to define $H$-pseudocharacters consisting of finitely many functions $\Gamma^I \ra k$ satisfying finitely many relations.

In Section \ref{examples}, we use invariant theoretic-results of Procesi and others \cite{procesi, aslaksen_tan_zhu, rogora} to give explicit finite presentations for the $\FFG$-algebras corresponding to the general and ordinary orthogonal groups $GO_n$ and $O_n$, the general and ordinary symplectic groups $GSp_n$ and $Sp_n$, and the special orthonal group $SO_n$.  By extension, we define explicit pseudocharacters for these groups.

Finally, in Section \ref{conjugacy}, we use our pseudocharacters to investigate the problem of conjugacy vs.\ element-conjugacy for representions $\Gamma \ra H$, when $H$ is a linear algebraic group for which one can define pseudocharacters.  We formulate a general condition in terms of $\FFG$-algebras under which element-conjugacy implies conjugacy.  We then use our explicit pseudocharacters for $GO_n(\C)$, $O_n(\C)$, $GSp_{2n}(\C)$, $Sp_{2n}(\C)$ to prove that for any group $\Gamma$, element-conjugate semisimple representations from $\Gamma$ to one of those groups are automatically conjugate.  Previous results of this form were only known for $O_n(\C)$ and $Sp_{2n}(\C)$, and only for compact $\Gamma$.  We also give a counterexample to the corresponding claim for $SO_{2n}(\C)$ ($n\ge 3$) which is simpler than that used in \cite[Proposition 3.8]{larsen_1}, and which extends that result to $SO_6(\C)$.

\section{General Results on Pseudocharacters}\label{general_results}
In this section, we define $\FFG$-algebras and use them to reformulate V.\ Lafforgue's result.  Section \ref{algebras} introduces $\FFG$-algebras and the closely related $\FI$-algebras and $\FFS$-algebras, modeled after the $\FI$-modules defined in \cite{fi_modules}.  Section \ref{lafforgue} restates V.\ Lafforgue's result in terms of morphisms between particular $\FFG$-algebras.  Finally, Section \ref{finiteness} shows that the $\FFG$-algebra $k[H^\bullet]^{\Ad H}$ (see Example \ref{invariant_example}) appearing in Theorem \ref{lafforgue_thm} is ``finitely presented'' as an $\FFG$-algebra, in an appropriate sense (in fact, it is finitely presented even as an $\FI$-algebra), and we explain how this finite presentation implies the general existence of ``finite'' pseudocharacters.

\subsection{FI-, FFS-, and FFG-algebras.}\label{algebras}
We denote by FI the category of finite sets, FFS the category of free finitely generated semigroups, and FFG the category of free finitely generated groups.  For every finite (nonempty) set $I$, let $\FS(I)$ (resp.\ $\FG(I)$) denote the free semigroup (resp.\ group) generated by $I$.

\begin{mylemma}
The category FFS is generated by the following two types of morphisms:
\begin{itemize}
  \item morphisms $\FS(I) \ra \FS(J)$ that sends generators to generators, i.e., those induced by maps between finite sets $I \ra J$
  \item morphisms
  \begin{align*}
  &\FS(\{x_1, \dots, x_n\}) \ra \FS(\{y_1, \dots, y_{n+1}\}), &x_i \mapsto y_i (i < n), x_n \mapsto y_ny_{n+1}.
  \end{align*}
\end{itemize}
The category FFG is generated by the above two types of morphisms (with $\FS$ replaced by $\FG$) together with:
\begin{itemize}
  \item morphisms
  \begin{align*}
  &\FG(\{x_1, \dots, x_n\}) \ra \FG(\{y_1, \dots, y_n\}), &x_i \mapsto y_i (i < n), x_n \mapsto y_n^{-1}.
  \end{align*}
\end{itemize}
\end{mylemma}

\begin{mydef}
Fix a commutative ring $k$.  An \emph{FI-algebra} (resp.\ \emph{FFS-algebra, FFG-algebra}) is a functor from FI (resp.\ FFS, FFG) to the category of $k$-algebras.  Morphisms between FI-algebras (resp.\ FFS-algebras, FFG-algebras) are natural transformations of functors.
\end{mydef}

If $A^\bullet$ is an FI-algebra (resp.\ FFS-algebra, FFG-algebra) and $I$ is a finite set, we will use $A^I$ to denote the $k$-algebra corresponding to $I$ under $A^\bullet$, and similary for morphisms $\Theta^\bullet: A^\bullet \ra B^\bullet$.  If $\phi: I \ra J$ (resp.\ $\FS(I) \ra \FS(J)$, $\FG(I) \ra \FG(J)$) is a morphism, then we will use $A^\phi$ to denote the corresponding $k$-algebra morphism $A^I \ra A^J$.

We can define kernels, cokernels, subobjects, quotients, and tensor products over $k$ in the category of FI-algebras (resp.\ FFS-algebras, FFG-algebras) by using the analogous constructions in the category of $k$-algebras, applying those constructions to each $k$-algebra in the image of an FI-algebra.  We say that a morphism $\Theta^\bullet$ is injective (resp.\ surjective, bijective) if each $\Theta^I$ has that property.

\begin{myrmk}
Any FFG-algebra is naturally an FFS-algebra, and any FFS-algebra is naturally an FI-algebra.  A morphism of FFG-algebras is also a morphism of FFS-algebras, and a morphism of FFS-algebras is also a morphism of FI-algebras.
\end{myrmk}

\begin{myex}
Let $\Gamma$ be a group and $R$ be a $k$-algebra.  We define an $\FFG$-algebra $\Map(\Gamma^\bullet, R)$ as follows.  To the finite set $I$, we associate $\Map(\Gamma^I, R)$, the $k$-algebra of all set maps $\Gamma^I \ra R$.  Next, recall that for any finite set $I$, $\Gamma^I = \Hom(\FG(I), \Gamma)$.  Thus for any group homomorphism $\phi: \FG(I) \ra \FG(J)$, we have a natural set map $\Gamma^J \ra \Gamma^I$, which induces a $k$-algebra morphism $\Map(\Gamma^I, R) \ra \Map(\Gamma^J, R)$; we associate this morphism to $\phi$.
\end{myex}

\begin{myex}\label{invariant_example}
Let $V$ be an affine variety over $k$, and let $H$ be a group which acts on $V$.  We define the $\FI$-algebra $k[V^\bullet]^H$ by the association $I \mapsto k[V^I]^H$, where $H$ acts diagonally on $V^I$.  For any set map $\phi: I \ra J$, we get a variety map $V^J \ra V^I$ defined over $k$, and this induces a $k$-algebra morphism $k[V^I]^H \ra k[V^J]^H$, which we associate to $\phi$.  If $V$ is also an algebraic semigroup (resp.\ group) whose multiplication is compatible with the action of $H$, then we can similarly give $k[V^\bullet]^H$ a structure of $\FFS$-algebra (resp.\ $\FFG$-algebra).
\end{myex}

For the remainder of this section, we state definitions and claims for FI-algebras, but they easily generalize to FFS-algebras and FFG-algebras.

\begin{mydef}
Let $A^\bullet$ be an FI-algebra.  Given a subset $\Sigma \subset \sqcup_I A^I$, the \emph{FI-algebra span ${\rm span}_\FI(A^\bullet, \Sigma)$ of $\Sigma$ in $A^\bullet$} is defined to be the minimum sub-FI-algebra of $A^\bullet$ containing each element of $\Sigma$.  An FI-algebra is \textit{finitely generated} if it equals the span of some finite set.
\end{mydef}

There is another way to characterize finite generation, in terms of free FI-algebras.  Let $m$ be a nonnegative integer, and let $\mbf{m} := \{1, 2, \dots, m\}$ denote the typical set of $m$ elements.

\begin{mydef}
The \textit{free FI-algebra of degree $m$}, denoted $F_\FI(m)^\bullet$, is defined by
\begin{align*}
F_\FI(m)^I &:= k[\{x_\psi | \psi \in \Hom_\FI(\mbf{m}, I)\}] \\
F_\FI(m)^\phi &:= (x_\psi \mapsto x_{\phi \circ \psi}).
\end{align*}
In the case of $\FFS$-algebras (resp.\ $\FFG$-algebras), we replace $\Hom_{\FI}(\mbf{m}, I)$ with $\Hom_{\FFS}(\FS(\mbf{m}), \FS(I))$ (resp.\ $\Hom_{\FFG}(\FG(\mbf{m}), \FG(I))$).
\end{mydef}

If $A^\bullet$ is an FI-algebra and $a \in A^{\mbf{m}}$, then it is easy to see that $x_{id_{\mbf{m}}} \mapsto a$ extends to a unique map of FI-algebras $F_\FI(m)^\bullet \ra A^\bullet$, and its image is precisely $\mbox{span}_\FI(A^\bullet, a)$.  Thus:

\begin{myprop}
An FI-algebra $A^\bullet$ is finitely generated iff it admits a surjective morphism $\bigotimes_i F_\FI(m_i) \ra A^\bullet$ for some finite sequence of integers $(m_i)$.
\end{myprop}

\begin{mydef}
Let $A^\bullet$ be an FI-algebra. An \textit{FI-ideal of $A^\bullet$} is an association $\mf{a}^\bullet$ taking each finite set $I$ to an ideal $\mf{a}^I$ of $A^I$, such that for all morphisms $\phi \in \Hom_\FI(I, J)$, we have $A^\phi(\mf{a}^I) \subset \mf{a}^J$.  Given a morphism of FI-algebras $\Theta^\bullet: A^\bullet \ra B^\bullet$, we define the \textit{kernel} of $\Theta^\bullet$ to be the association $\ker(\Theta^\bullet)$ taking each finite set $I$ to the ideal $\ker(\Theta^I : A^I \ra B^I)$ of $A^I$.  We define the \textit{radical} of an $\FI$-ideal $\mf{a}^\bullet$ to be the association $I \mapsto \sqrt{\mf{a}^I}$, where the radical is taken in $A^I$.
\end{mydef}

The following lemma is easy.

\begin{mylemma}
\begin{enumerate}[(i)]
  \item Let $\Theta^\bullet: A^\bullet \ra B^\bullet$ be a morphism of FI-algebras.  Then $\ker(\Theta^\bullet)$ is an FI-ideal of $A^\bullet$.
  \item Let $\mf{a}^\bullet$ be an FI-ideal of $A^\bullet$.  Then there exists an FI-algebra $B^\bullet$ and a surjective morphism $\Theta^\bullet: A^\bullet \ra B^\bullet$ such that $\ker(\Theta^\bullet) = \mf{a}^\bullet$.  Furthermore, the pair $(B^\bullet, \Theta^\bullet)$ is unique up to unique isomorphism.
\end{enumerate}
\end{mylemma}

We denote the FI-algebra in part (ii) by $A^\bullet / \mf{a}^\bullet$.

\begin{mydef}
Let $A^\bullet$ be an FI-algebra.  Given a subset $\Sigma \subset \sqcup_I A^I$, we define the \textit{FI-ideal generated by $\Sigma$} to be the minimum FI-ideal of $A^\bullet$ containing each element of $\Sigma$.  We define an FI-ideal to be \textit{finitely generated} if it is generated by some finite set.
\end{mydef}

\subsection{Lafforgue's theorem.}\label{lafforgue}
Let $H$ be a connected reductive group defined over $k$, and let $\Gamma$ be an abstract group.  For any finite set $I$, we let $\Ad H$ denote the diagonal conjugation action of $H$ on $H^I$, and we let $k[H^\bullet]^{\Ad H}$ denote the $\FFG$-algebra in Example \ref{invariant_example} corresponding to this action.  Also let $\Map(\Gamma^\bullet/\Ad \Gamma, k)$ be the sub-$\FFG$-algebra of $\Map(\Gamma^\bullet, k)$ consisting of functions with are invariant under diagonal conjugation by $\Gamma$. Then we can rephrase V.\ Lafforgue's result as follows.

\begin{mythm}[{\cite[Proposition 11.7]{lafforgue_french}}]\label{lafforgue_thm}
Let $H$ be a connected reductive group defined over $k$.  Assume $\charac k = 0$.  Then:
\begin{enumerate}[(i)]
  \item There is a natural bijection between
  \[
  \{\mbox{$H(\overline{k})$-conjugacy classes of semisimple representations $\rho: \Gamma \ra H(\overline{k})$}\}
  \]
  and $\FFS$-algebra morphisms
  \[
  \Theta^\bullet: \overline{k}[H^\bullet]^{\Ad H} \ra \Map(\Gamma^\bullet/\Ad \Gamma, \overline{k}).
  \]
  The bijection is given by sending $\rho: \Gamma \ra H(\overline{k})$ to the $\FFS$-algebra morphism $\Theta^\bullet$ given by
  \[
  \Theta^{\mbf{n}}(f)(\gamma_1, \dots, \gamma_n) = f(\rho(\gamma_1), \dots, \rho(\gamma_n)).
  \]
  
  \item In (i), if $\Theta^\bullet$ restricts to give an $\FFS$-algebra morphism $k[H^\bullet]^{\Ad H} \ra \Map(\Gamma^\bullet/\Ad \Gamma, k)$, then the corresponding conjugacy class contains a representation $\rho: \Gamma \ra H(k')$ for some finite extension $k'/k$.
  
  \item If $\Gamma$ is profinite, $H$ is split over $k$, and $k$ is a finite extension of $\Q_l$ for some $l$, then (i) and (ii) hold with ``representation'' replaced by ``continuous representation'' and with $\Map(\Gamma^\bullet/\Ad \Gamma, \overline{k})$ replaced by the $\FFG$-algebra $C(\Gamma^\bullet/\Ad \Gamma, \overline{k})$ of continuous $\Ad \Gamma$-invariant maps $\Gamma^I \ra \overline{k}$ (and similarly for $\Map(\Gamma^\bullet/\Ad \Gamma, k)$).
\end{enumerate}
\end{mythm}

\begin{mycor}\label{lafforgue_cor}
The above theorem holds with $\FFS$-algebra morphisms replaced by $\FFG$-morphisms.
\end{mycor}
\begin{proof}
Any $\FFG$-algebra morphism $\overline{k}[H^\bullet]^{\Ad H} \ra \Map(\Gamma^\bullet/\Ad \Gamma, \overline{k})$ is also an $\FFS$-algebra morphism.  Conversely, given an $\FFS$-algebra morphism $\Theta^\bullet: \overline{k}[H^\bullet]^{\Ad H} \ra \Map(\Gamma^\bullet/\Ad \Gamma, \overline{k})$, we get a representation $\rho: \Gamma \ra H(\overline{k})$ by the theorem, and then the relation
\[
\Theta^{\mbf{n}}(f)(\gamma_1, \dots, \gamma_n) = f(\rho(\gamma_1), \dots, \rho(\gamma_n))
\]
shows that $\Theta^\bullet$ is in fact an $\FFG$-algebra morphism.
\end{proof}

\subsection{Explicit descriptions of pseudocharacters.}\label{finiteness}
In this section, we will show that whenever Lafforgue's theorem applies to $H$, the $\FFG$-algebra $k[H^\bullet]^{\Ad H}$ is ``finitely presented'' in an appropriate sense.  In fact, this is true even of $k[H^\bullet]^{\Ad H}$ as an $\FI$-algebra.  As a consequence, it is always possible to define pseudocharacters for $H$ very explicitly, in a sense which will be made clear in Section \ref{examples}.

\begin{mythm}
Assume $k$ is a field of characteristic 0.  Let $H$ be a reductive group over $k$ which acts linearly on a finite-dimensional $k$-vector space $V$.  Then the FI-algebra $k[V^\bullet]^H$ is finitely generated.
\end{mythm}
\begin{proof}
By the Hilbert-Nagata theorem, for every finite set $I$, $k[V^I]^H$ is finitely generated as a $k$-algebra.  Let $d = \dim V$, and let $\Omega$ be a finite set of multihomogenous $k$-algebra generators for $k[V^d]^H$.  Then by \cite[Theorem 11.1.1.1]{procesi_book}, for all $n$, $k[V^n]^H$ is generated by polarizations of elements of $k[V^d]^H$, from which one can see that $k[V^n]^H$ is generated by polarizations of elements of $\Omega$.  In other words, $\varinjlim_n k[V^n]^H$ is generated by polarizations of elements of $\Omega$.

Now easily any polarization of a multihomogeneous function $h$ can be obtained as a further polarization of any full polarization of $h$ (up to a scalar multiple), since $\charac k = 0$.  So, letting $\Sigma$ be a finite set containing one full polarization of each element of $\Omega$, $\varinjlim_n k[V^n]^H$ is also generated by $\Sigma$ under polarization.

Now let $f \in k[V^n]^H$ for some $n$.  We claim that $f$ is in the $\mbf{n}$-part of $\mbox{span}_{\FI}(k[V^\bullet]^H, \Sigma)$.  By the above paragraph, there are elements $g_1, \dots, g_r \in \Sigma$ with polarizations $g_1^1, \dots, g_1^{i_1}, \dots, g_r^1, \dots, g_r^{i_r}$ such that $f$ is in the $k$-algebra generated by the $g_j^l$.  Since any further polarization of a full polarization results from vector variable substitutions in the full polarization, we see that each $g_j^l$ is in the $\FI$-algebra span of $\Sigma$.  Using the natural embeddings $k[V^s]^H \subset k[V^t]^H$ whenever $s \le t$, we can assume that all $g_j^l$ are in $k[V^N]^H$ for some $N$.  Then the image of $f$ under the embedding $k[V^n]^H \subset k[V^N]^H$ lies in the $k$-algebra generated by the $g_j^l$.  Mapping this image of $f$ back to $k[V^n]^H$ using the map $(k[V^\bullet]^H)^\phi$ for some $\phi: \mbf{N} \ra \mbf{n}$ which is the identity on $\mbf{n}$, we thus find that $f$ is in the FI-algebra span of $\Sigma$.
\end{proof}

\begin{mylemma}\label{lem_quotient}
Assume $k$ is a field of characteristic 0.  Let $H$ be a reductive group over $k$ which acts rationally on a finitely generated $k$-algebra $R$, and let $J$ be an ideal of $R$ invariant under $H$.  Then $(R/J)^H = R^H/(R^H \cap J)$.
\end{mylemma}
\begin{proof}
The action of $H$ on $R$ is completely reducible by the hypotheses, so we can find a $k[H]$-module complement to $J$ in $R$; call it $C$.  Easily the projection $R \epi C$ restricts to give a surjective map $R^H \epi C^H$.  Then the identical map $R^H \ra (R/J)^H$ is also surjective, obviously with kernel $R^H \cap J$.
\end{proof}

\begin{mycor}
Assume $k$ is a field of characteristic 0.  Let $H$ be a reductive linear algebraic group over $k$.  Then the $\FFG$-algebra $k[H^\bullet]^{\Ad H}$ is finitely generated as an $\FI$-algebra, hence also as an $\FFS$- and $\FFG$-algebra.
\end{mycor}
\begin{proof}
Let $d$ be such that $H$ is an affine sub-group variety of $GL_d$ over $k$.  Using the embedding of $GL_d$ into $M_d \times \mb{A}^1$ given by sending $A \in GL_d(k)$ to $(A, \det(A)^{-1})$, we can consider $GL_d$, hence $H$, as an affine subvariety of $M_d \times \mb{A}^1$.  Then $H$ acts linearly on $M_d(k) \times \mb{A}^1(k)$ (as a $k$-vector space) by conjugation on $M_d(k)$; call this action $\Ad H$.  Obviously this action restricts to give the conjugation action $\Ad H$ of $H$ on itself.

By the theorem, $k[M_d \times \mb{A}^1]^{\Ad H}$ is finitely generated as an $\FI$-algebra.  Now let $\mf{a} \subset k[M_d \times \mb{A}^1]$ be the ideal which cuts out $H$ as a variety.  Then by Lemma \ref{lem_quotient}, for all finite sets $I$,
\[
k[H^I]^{\Ad H} = \frac{k[(M_d \times \mb{A}^1)^I]^{\Ad H}}{\mf{a}^I \cap k[(M_d \times \mb{A}^1)^I]^{\Ad H}}.
\]
Hence we can exhibit $k[H^\bullet]^{\Ad H}$ as a quotient of $k[(M_d \times \mb{A}^1)^\bullet]^{\Ad H}$, and obviously any quotient of a finitely generated FI-algebra is finitely generated.
\end{proof}

Recall that if $A^\bullet$ is finitely generated as an $\FI$-algebra, then there is a surjective morphism $\bigotimes_i F_{\FI}(m_i) \ra A^\bullet$ for some finite sequence of integers $(m_i)$.  We now show that the kernel of such a morphism is always finitely generated as an $\FI$-ideal, hence any finitely generated $\FI$-algebra is actually finitely presented.  To do this, we prove a statement analogous to the Noetherian property of polynomial rings over $k$, as follows.

\begin{myprop}
Let $(m_i)$ be a finite sequence of integers, and let $A^\bullet = \otimes_i F_{\FI}(m_i)$.  Let $\mf{a}^\bullet$ be an FI-ideal of $A^\bullet$.  Then $\mf{a}^\bullet$ is finitely generated.
\end{myprop}
\begin{proof}
Let $M = \max\{m_i\}$.  For fixed $i$, let $\iota: \mbf{m_i} \ra \mbf{M}$ be the canonical injection, and define a morphism $\Theta^\bullet: F_{\FI}(M) \ra F_{\FI}(m_i)$ by sending $x_{id_{\mbf{M}}}$ to $x_\iota$.  For any $\FI$-ideal $\mf{b}^\bullet$ of $F_{\FI}(m_i)$, we can define an $\FI$-ideal $(\Theta^\bullet)^{-1}(\mf{b}^\bullet)$ of $F_{\FI}(M)$ by setting $\left((\Theta^\bullet)^{-1}(\mf{b}^\bullet)\right)^I = (\Theta^I)^{-1}(\mf{b}^I)$.  Let $\pi: \mbf{M} \ra \mbf{m_i}$ be some map such that $\pi \circ \iota = id_{\mbf{m_i}}$; then $x_{id_{\mbf{m_i}}} = (F_{\FI}(m_i))^\pi(x_\iota)$, so $\Theta^\bullet$ is surjective.  Hence if $(\Theta^\bullet)^{-1}(\mf{b}^\bullet)$ is finitely generated, then so is $\mf{b}^\bullet$.

Thus WLOG all $m_i = M$ for some integer $M$.  Then $A^\bullet = F_{\FI}(M)^{\otimes n}$ for some integer $n$.  Now let $\mf{b}^\bullet$ be the $\FI$-ideal of $A^\bullet$ generated by $\mf{a}^{\mbf{M}}$.  
Then the identity maps $(A^\bullet/\mf{a}^\bullet)^{\mbf{M}} \ra (A^\bullet/\mf{b}^\bullet)^{\mbf{M}}$, $(A^\bullet/\mf{b}^\bullet)^{\mbf{M}} \ra (A^\bullet/\mf{a}^\bullet)^{\mbf{M}}$ induce maps $A^\bullet/\mf{a}^\bullet \ra A^\bullet/\mf{b}^\bullet$, $A^\bullet/\mf{b}^\bullet \ra A^\bullet/\mf{a}^\bullet$ which are inverses to each other and which agree with the projections $A^\bullet \ra A^\bullet/\mf{a}^\bullet$, $A^\bullet \ra A^\bullet/\mf{b}^\bullet$, by properties of free $\FI$-algebras.  Hence $\mf{a}^\bullet = \mf{b}^\bullet$, so $\mf{a}^\bullet$ is generated by $\mf{a}^{\mbf{M}}$ as an $\FI$-ideal.

From the definition, $A^{\mbf{M}}$ is a finitely generated $k$-algebra, hence is Noetherian.  Thus $\mf{a}^{\mbf{M}}$ is finitely generated as an ideal in $A^{\mbf{M}}$.  Any finite set of generators then provides a finite set of generators for $\mf{a}^\bullet$.
\end{proof}

Now for fixed $H$, by choosing a finite set of generators for $k[H^\bullet]^{\Ad H}$ as an $\FFG$-algebra, as well as a finite set of generators for the $\FFG$-ideal of relations between those generators (or even a set of generators up to radical), we can define pseudocharacters for $H$ in terms of a finite set of functions satisfying finitely many relations.  This technique was first demonstrated in \cite[Remark 11.8]{lafforgue_french}, wherein V.\ Lafforgue implicitly gives a finite presentation for $k[GL_n^\bullet]^{\Ad GL_n}$ and explains how it implies Taylor's original result on $GL_n$-pseudocharacters.  We further illustrate the technique with examples in Section \ref{examples} below.

\section{Explicit Pseudocharacters for Classical Groups}\label{examples}

\subsection{(General) Orthogonal Group}
We now present new results which establish pseudocharacters for the orthogonal and general orthogonal groups.  Assume $k$ is a field of characteristic 0.

Let $GO_n(k) = \{A \in M_n(k) \mid \mbox{for some $\lambda \in k^\times$, $AA^t = \lambda I$}\}$ be the $n$-dimensional general orthogonal group.  It is a connected reductive algebraic group, and it is in fact an affine subvariety of $GL_n$, hence of $M_n \times \mb{A}^1$.  Define a function $\lambda: GO_n(k) \ra k$ by $AA^t = \lambda(A)I$.  Then $\lambda \in k[GO_n]^{\Ad GO_n}$.

\begin{myprop}\label{go_n_gens}
$k[GO_n^\bullet]^{\Ad GO_n}$ is generated as an FFG-algebra by $\tr$ and $\lambda$.
\end{myprop}
\begin{proof}
Since $GO_n(k) \supset O_n(k)$, $k[M_n^\bullet]^{\Ad GO_n} \subset k[M_n^\bullet]^{\Ad O_n}$.  By Procesi's results on the invariants of $O_n(k)$ acting on matrices by conjugation \cite[Theorem 7.1]{procesi}, for all $m$, $k[M_n^m]^{\Ad O_n}$ is generated as a $k$-algebra by invariants $\tr(M)$, where $M \in \FS(\{A_1, A_1^t, \dots, A_m, A_m^t\})$.  The $\tr(M)$ are obviously also $GO_n(k)$-invariants, so $k[M_n^m]^{\Ad GO_n}$ has the same generators.  Then $k[(M_n \times \mb{A}^1)^m]^{\Ad GO_n}$ is generated as a $k$-algebra by the $\tr(M)$ and by the coordinate functions for the $m$ copies of $\mb{A}^1$, which we will denote $\det^{-1}(A_1), \dots, \det^{-1}(A_m)$.  By Lemma \ref{lem_quotient}, $k[GO_n^m]^{\Ad GO_n}$ is a quotient of $k[(M_n \times \mb{A}^1)^m]^{\Ad GO_n}$ for all $m$, so it is also generated by the invariants $\tr(M)$ and $\det^{-1}(A_i)$.  Then using the identity $A^t = \lambda(A)A^{-1}$ for $A \in GO_n(k)$, we see that any invariant $\tr(M)$ is in the FFG-algebra generated by $\tr$ and $\lambda$.  Also, using the identity $\det^{-1}(A) = \det(A^{-1})$ and the fact that we can express $\det(A^{-1})$ in terms of $\tr(A^{-1}), \dots, \tr(A^{-n})$, we see that any invariant $\det^{-1}(A_i)$ is in the FFG-algebra generated by $\tr$.
\end{proof}

The relations between the invariants are more complicated to describe.  We first summarize Procesi's result on relations between the generators $\tr(M)$ of $k[M_n^m]^{\Ad O_n} = k[M_n^m]^{\Ad GO_n}$.

Let $R$ be the polynomial ring over $k$ with indeterminates $T_M$ as $M$ varies over $\FS(\{A_1, A_1^t, \dots, A_m, A_m^t\})$, except that we make the identifications $T_{MN} = T_{NM}$ and $T_M = T_{M^t}$ for all words $M$ and $N$ (where $M^t$ is defined in the obvious way).  Let $\pi: R \ra k[M_n^m]^{\Ad GO_n}$ be the $k$-algebra homomorphism sending each $T_M$ to $\tr(M)$, which by \cite[Theorem 7.1]{procesi} is surjective.

Given $M_1, M_2, \dots, M_{n+1} \in \FS(\{A_1, A_1^t, \dots, A_m, A_m^t\})$ and an integer $0 \le j \le (n+1)/2$, define $F_{j, n+1}(M_1, M_2, \dots, M_{n+1}) \in R$ as follows.  Let $s$ be given by $n + 1 = 2j + s$.  Let $S$ be a set of formal symbols $(a, b)$, where each $a$ and $b$ is one of the formal symbols $u_1, \dots, u_{n+1}, v_1, \dots, v_{n+1}$.  Let $D^j = \sum_{\sigma \in S_{n+1}} \sgn(\sigma) D^j_\sigma$ be the following $(n+1) \times (n+1)$ determinant, as a function of symbols in $S$:
\[
\begin{vmatrix}
(u_1, u_{j+s+1}) & (u_1, u_{j+s+2}) & \cdots & (u_1, u_{n+1}) & (u_1, v_{j+1}) & (u_1, v_{j+2}) & \cdots & (u_1, v_{n+1}) \\
\vdots \\
(u_{j+s}, u_{j+s+1}) & (u_{j+s}, u_{j+s+2}) & \cdots & (u_{j+s}, u_{n+1}) & (u_{j+s}, v_{j+1}) & (u_{j+s}, v_{j+2}) & \cdots & (u_{j+s}, v_{n+1}) \\
(v_1, u_{j+s+1}) & (v_1, u_{j+s+2}) & \cdots & (v_1, u_{n+1}) & (v_1, v_{j+1}) & (v_1, v_{j+2}) & \cdots & (v_1, v_{n+1}) \\
\vdots \\
(v_j, u_{j+s+1}) & (v_j, u_{j+s+2}) & \cdots & (v_j, u_{n+1}) & (v_j, v_{j+1}) & (v_j, v_{j+2}) & \cdots & (v_j, v_{n+1})
\end{vmatrix}
\]
Next, using the formal identities $(a, b) = (b, a)$ and allowing the symbols $(a, b)$ to commute with each other, write each monomial $D^j_\sigma$ of $D^j$ in the form
\[
D^j_\sigma = (w_{i_1}, \bar{w}_{i_2})(w_{i_2}, \bar{w}_{i_3}) \dots (w_{i_j}, \bar{w}_{i_1}) \cdot (w_{j_1}, \bar{w}_{j_2})(w_{j_2}, \bar{w}_{j_3}) \dots (w_{j_s}, \bar{w}_{j_1}) \dots
\]
where $w_a$ stands for either $u_a$ or $v_a$, and by definition, $\bar{u}_a = v_a$ and $\bar{v}_a = u_a$.  Now define $T^j_\sigma(M_1, \dots, M_{n+1})$ by
\[
T^j_\sigma(M_1, \dots, M_{n+1}) = T_{N_{i_1}N_{i_2} \dots N_{i_j}} T_{N_{j_1}N_{j_2} \dots N_{j_j}} \dots,
\]
where $N_a = M_a$ or $N_a = M_a^t$, according to the inductively defined rules:
\begin{itemize}
  \item $N_{i_1} = M_{i_1}$, if $w_{i_1} = v_{i_1}$; else $N_{i_1} = M_{i_1}^t$
  \item Set $N_{i_{t+1}}$ to be the same type as $N_{i_t}$ (transposed or not transposed) if and only if $w_{i_t}$ and $w_{i_{t+1}}$ stand for instances of the same letter ($u$ or $v$).
\end{itemize}
Then
\[
F_{j, n+1}(M_1, \dots, M_{n+1}) := \sum_{\sigma \in S_{n+1}} \sgn(\sigma) T^j_\sigma(M_1, \dots, M_{n+1})
\]
is the result of replacing each $D^j_\sigma$ with $T^j_\sigma$ in $D^j$.  Because $T_{MN} = T_{NM}$ and $T_M = T_{M^t}$ by assumption, the functions $T^j_\sigma$ are well-defined, hence so is $F_{j, n+1}$.  Note that $F_{0, n+1}(M_1, \dots, M_{n+1})$ reduces to (\ref{gl_n_relation}), the non-trivial relation for $GL_n$-pseudocharacters.

\begin{mythm}[{\cite[Theorem 8.4(a)]{procesi}}]
$\ker(\pi)$ is the ideal of $R$ generated by the $F_{j, n+1}(M_1, \dots, M_{n+1})$, $0 \le j \le (n+1)/2$, as the $M_i$ vary over $\FS(\{A_1, A_1^t, \dots, A_m, A_m^t\})$.
\end{mythm}

Now let $\psi$ be the composition of $\pi$ with the map $k[M_n^m]^{\Ad GO_n} \mono k[(M_n \times \mb{A}^1)^m]^{\Ad GO_n} \epi k[GO_n^m]^{\Ad GO_n}$, which is still surjective by the proof of Proposition \ref{go_n_gens}.  Intuitively, one should expect $\ker(\psi)$ to be $\ker(\pi)$ plus the relations of the form $T_{MNN^tP} = T_{MP}T_{NN^t}$, since $GO_n$ is defined by the condition that $NN^t$ is a scalar matrix for all $N \in GO_n(k)$.  The next proposition shows that this is indeed the case, at least up to radical.

\begin{myprop}\label{prop_orthogonal_invariants}
$\ker(\psi)$ is the radical of the ideal generated by $\ker(\pi)$ and the relations $T_{MNN^tP} - T_{MP}T_{NN^t}$ for $M, N, P \in \FS(\{A_1, A_1^t, \dots, A_m, A_m^t\})$.
\end{myprop}
\begin{proof}
It suffices to show this for $\overline{k}$, so WLOG $k$ is algebraically closed.  Let $J \subset R$ be the ideal generated by $\ker(\pi)$ and the $T_{MNN^tP} - T_{MP}T_{NN^t}$.  It suffices to prove that $\pi(\ker(\psi)) = \sqrt{\pi(J)}$.  Now $\sqrt{\pi(\ker(\psi))} = \pi(\ker(\psi))$ because $k[M_n^m]^{\Ad GO_n}/\pi(\ker(\psi)) \cong k[GO_n^m]^{\Ad GO_n}$ is reduced, so by the Nullstellensatz, it suffices to prove that $\pi(\ker(\psi))$ and $\pi(J)$ define the same subvariety of $\Spec(k[M_n^m]^{\Ad GO_n})$.  Using the $\tr(M)$ as coordinate functions for $\Spec(k[M_n^m]^{\Ad GO_n})$, the subvariety associated to $\pi(\ker(\psi))$ is the set of all points of the form $(\tr(M\{A_i \mapsto B_i\}))_{M \in \FS\{A_1, A_1^t, \dots, A_m, A_m^t\}}$ for some $B_1, \dots, B_m \in GO_n(k)$, where $M\{A_i \mapsto B_i\}$ denotes the element of $GO_n(k)$ obtained by substituting each $A_i$ for $B_i$ in $M$.  Meanwhile, the subvariety associated to $\pi(J)$ is the set of all points of the form $(\tr(M\{A_i \mapsto C_i\}))_{M \in \{A_1, A_1^t, \dots, A_m, A_m^t\}}$ where $C_1, \dots, C_m \in M_n(k)$ are such that $\tr(MNN^tP) = \tr(MP)\tr(NN^t)$ whenever $M$, $N$, and $P$ are semigroup words in the $C_i$ and $C_i^t$.  The following lemma shows that these two subvarieties are equal, proving the claim.
\end{proof}

\begin{mylemma}\label{lem_orthogonal_matrices}
Let $C_1, \dots, C_m \in M_n(\overline{k})$ be such that $\tr(MNN^tP) = \tr(MP)\tr(NN^t)$ whenever $M$, $N$, and $P$ are semigroup words in the $C_i$ and $C_i^t$.  Then there exist $B_1, \dots, B_m \in GO_n(\overline{k})$ such that for all $M \in \FS(\{A_1, A_1^t, \dots, A_m, A_m^t\})$, $\tr(M\{A_i \mapsto B_i\}) = \tr(M\{A_i \mapsto C_i\})$.
\end{mylemma}
\begin{proof}
Let $(V, B)$ be the bilinear space with $V \cong \overline{k}^n$ and $B$ the standard nondegenerate symmetric bilinear form, i.e., the dot product.  Let $A$ be the noncommutative $\overline{k}$-algebra
\[
A := \overline{k}[C_1, \dots, C_r, C_1^t, \dots, C_r^t] \subset M_n(\overline{k}),
\]
which has the natural involution $(-)^t$.  Then the natural representation $\rho: A \mono M_n(\overline{k}) \cong \End(V, B)$ is orthogonal, i.e., it preserves involutions.

Then by \cite[Theorem 15.2(b)(c)]{procesi} and the fact that all nondegenerate bilinear forms on $V$ are equivalent, there exists a semisimple orthogonal representation $\rho^{ss}: A \ra \End(V)$ such that $\tr(\rho) = \tr(\rho^{ss})$.  Thus setting $B_i = \rho^{ss}(C_i)$, we will be done once we prove that $B_i \in GO_n(\overline{k})$.  Now for any $D \in A$, we have
\begin{align*}
\tr((B_iB_i^t - \tr(B_iB_i^t)I)\rho^{ss}(D))
&= \tr(\rho^{ss}((C_iC_i^t - \tr(C_iC_i^t)I)D)) \\
&= \tr(\rho((C_iC_i^t - \tr(C_iC_i^t)I)D)) \\
&= \tr((C_iC_i^t)D) - \tr(C_iC_i^t)\tr(D) \\
&= 0
\end{align*}
by assumption.  Since $\rho^{ss}$ is semisimple, $\tr$ is a nondegenerate bilinear form on $\Ima(\rho^{ss})$, so this shows that $B_iB_i^t - \tr(B_iB_i^t)I = 0$, proving the claim.
\end{proof}

Next, let $S$ be the polynomial ring over $k$ with indeterminates:
\begin{itemize}
  \item $U_M$ for $M \in \FG(\{A_1, \dots, A_m\})$, with the identifications $U_1 = n$ and $U_{MN} = U_{NM}$ for all words $M$, $N$
  \item $l_M$ for $M \in \FG(\{A_1, \dots, A_m\})$, with the identifications $l_1 = 1$ and $l_{MN} = l_M l_N$ for all words $M$, $N$,
\end{itemize}
and let $J \subset S$ be the ideal generated by relations of the form $U_M - l_M U_{M^{-1}}$.  Then we have a surjective map $\rho: S/J \epi k[GO_n^m]^{\Ad GO_n}$ defined by $\rho(U_M) = \tr(M)$ and $\rho(l_M) = \lambda(M)$.

Now easily $\psi$ factors through $\rho$ via the map $\tau: R \ra S/J$ which sends $T_M$ to $l_{M'}U_{M''}$, where $M'$ is the product (with multiplicity) of all letters $A_1, \dots, A_n$ which appear transposed in $M$, and $M''$ is the result of substituting all transposed letters $A_i^t$ in $M$ with $A_i^{-1}$.  From this and the above proposition, noting that $\tau(T_{MNN^tP} - T_{MP}T_{NN^t}) = 0$, we get:

\begin{mycor}
$\ker(\rho: S \ra k[GO_n^m]^{\Ad GO_n})$ is the radical of the ideal generated by the relations:
\begin{itemize}
  \item $U_M - \lambda_M U_{M^{-1}}$ for $M \in \FG(\{A_1, \dots, A_m\})$
  \item $G_{j, n+1}(M_1, \dots, M_{n+1})$, $0 \le j \le (n+1)/2$, as the $M_i$ vary over words in $\FS(\{A_1, A_1^t, \dots, A_m, A_m^t\})$.  Here $G_{j, n+1}(M_1, \dots, M_{n+1})$ is the same as $F_{j, n+1}(M_1, \dots, M_{n+1})$ except that we replace each $T_M$ with $l_{M'}U_{M''}$, where $M'$ is the product (with multiplicity) of all letters $A_1, \dots, A_m$ which appear transposed in $M$, and $M'' \in \FG(\{A_1, \dots, A_m\})$ is the result of substituting all transposed letters $A_i^t$ in $M$ with $A_i^{-1}$.
\end{itemize}
\end{mycor}

Note that $\tau(F_{j, n+1}(M_1, \dots, M_{n+1}))$ equals $G_{j, n+1}(M_1, \dots, M_{n+1})$ modulo $J$.

Finally, we get a finite presentation for $k[GO_n^\bullet]^{\Ad GO_n}$ as an FFG-algebra.

\begin{mycor}
Let $N \in \N$ be such that $k[GO_n^\bullet]^{\Ad GO_n}$ is generated by $k[GO_n^N]^{\Ad GO_n}$ as an FI-algebra and $N \ge n+1$.  Let $A^\bullet$ be the free FFG-algebra on letters $T, l$ in degree $N$.  Then the FFG-algebra map $\Theta^\bullet: A^\bullet \ra k[GO_n^\bullet]^{\Ad GO_n}$ sending $T$ to $\tr(A_1)$ and $l$ and $\lambda(A_1)$ is surjective.  Denote the generators of $\FG(\mbf{N})$ by $g_1, \dots, g_N$, and for $g \in \FG(\mbf{N})$, let $\phi_g$ denote some fixed map $\FG(\mbf{N}) \ra \FG(\mbf{N})$ sending $g_1$ to $g$.  Then the kernel of $\Theta^\bullet$ is the radical of the FFG-algebra ideal generated by the relations:
\begin{itemize}
  \item $A^{\phi_1}(T) - n$
  \item $A^{\phi_{g_1g_2}}(T) - A^{\phi_{g_2g_1}}(T)$
  \item $A^{\phi_1}(l) - 1$
  \item $A^{\phi_{g_1g_2}}(T) - A^{\phi_{g_1}}(T)A^{\phi_{g_2}}(T)$
  \item $T - lA^{\phi_{g_1^{-1}}}(T)$
  \item $G_{j, n+1}(g_1, \dots, g_{n+1})$, $0 \le j \le (n+1)/2$, which is defined in the same way as $G_{j, n+1}(M_1, \dots, M_{n+1})$ by abuse of notation, except that we replace each symbol $l_g$ with $A^{\phi_g}(l)$ and each $T_g$ with $A^{\phi_g}(T)$.
\end{itemize}
\end{mycor}

We now apply Lafforgue's result.

\begin{mydef}
Let $\Gamma$ be a group.  A \emph{$GO_n$-pseudocharacter of $\Gamma$ over $k$} is a pair $(T, l)$, consisting of a set map $T: \Gamma \ra k$ and a group homomorphism $l: \Gamma \ra k^\times$, such that
\begin{itemize}
  \item $T(1) = n$
  \item For all $\gamma_1, \gamma_2 \in \Gamma$, $T(\gamma_1\gamma_2) = T(\gamma_2\gamma_1)$
  \item For all $\gamma \in \Gamma$, $T(\gamma) = l(\gamma)T(\gamma^{-1})$
  \item For all integers $0 \le j \le (n+1)/2$ and for all $\gamma_1, \dots, \gamma_{n+1} \in \Gamma$, $T$ and $l$ satisfy the relation $H_{j, n+1}(l, T, \gamma_1, \dots, \gamma_{n+1}) = 0$, where $H_{j, n+1}(l, T, \gamma_1, \dots, \gamma_{n+1})$ is defined in the same way as $G_{j, n+1}(\gamma_1, \dots, \gamma_{n+1})$ by abuse of notation, except that we replace each symbol $l_\gamma$ with $l(\gamma)$ and each $T_\gamma$ with $T(\gamma)$.
\end{itemize}
\end{mydef}

\begin{mydef}
An \emph{$O_n$-pseudocharacter of $\Gamma$ over $k$} is a set map $T: \Gamma \ra k$ such that $(T, 1)$ is a $GO_n$-pseudocharacter.
\end{mydef}

\begin{mythm}\label{gon_main_thm}
\begin{enumerate}[(i)]
  \item There is a natural bijection between $GO_n(\overline{k})$-conjugacy classes of semisimple representations $\rho: \Gamma \ra GO_n(\overline{k})$ and $GO_n$-pseudocharacters $(T, l)$ of $\Gamma$ over $\overline{k}$.
  The bijection is given by sending $\rho: \Gamma \ra GO_n(\overline{k})$ to $(\tr(\rho), \lambda(\rho))$, where $\lambda(\rho)(\gamma) := \lambda(\rho(\gamma))$.
  
  \item There is a natural bijection between $O_n(\overline{k})$-conjugacy classes of semisimple representations $\rho: \Gamma \ra O_n(\overline{k})$ and $O_n$-pseudocharacters $T$ of $\Gamma$ over $\overline{k}$.
  The bijection is given by sending $\rho: \Gamma \ra O_n(\overline{k})$ to $\tr(\rho)$.
  
  \item If $(T, l)$ (resp.\ $T$) is a $GO_n$- (resp.\ $O_n$-) pseudocharacter over $k$, then the corresponding conjugacy class over $\overline{k}$ given by (i) (resp.\ (ii)) contains a representation $\rho: \Gamma \ra GO_n(k')$ (resp.\ $\rho: \Gamma \ra O_n(k')$) for some finite extension $k'/k$.
  
  \item If $\Gamma$ is profinite and $k$ is a complete extension of $\Q_l$ for some $l$, then (i) and (ii) hold with ``representation'' replaced by ``continuous representation'' and with ``$GO_n$-pseudocharacters'' replaced by ``continuous $GO_n$-pseudocharacters''.  Here $(T, l)$ is called continuous if both $T$ and $l$ are continuous.
\end{enumerate}
\end{mythm}

Note that we get the above result for $O_n$ even though it is not connected.

\begin{myrmk}
The above results can also be proven by modifying Taylor's proof for $GL_n$-pseudocharacters \cite[Theorem 1]{taylor}.  In fact, one can generalize the above result to algebras, as follows.  First define a $*$-algebra to be a (possibly noncommutative) $\overline{k}$-algebra with an involution $*$.  Define an orthogonal $n$-dimensional representation of a $*$-algebra $R$ to be a $\overline{k}$-algebra morphism $R \ra M_n(\overline{k})$ mapping $*$ to the transpose.  Then one can define $n$-dimensional orthogonal pseudocharacters of a $*$-algebra $R$ similarly to the definition of $O_n$-pseudocharacters above.  Using \cite[Theorem 15.3]{procesi} in place of \cite[Lemma 2]{taylor} in Taylor's proof, one can prove that these are in bijection with $O_n(\overline{k})$-conjugacy classes of semisimple orthogonal representations of $R$.  By taking $R$ to be the group algebra $\overline{k}[\Gamma]$ with involution determined by $(\gamma)^* = l(g)(g^{-1})$ for $\gamma \in \Gamma$, one recovers Theorem \ref{gon_main_thm}.
\end{myrmk}

\subsection{(General) Symplectic Group}
Again assume $k$ is a field of characteristic 0.  Let $GSp_{2n}(k) = \{A \in M_{2n}(k) \mid \mbox{for some $\lambda \in k^\times$, $AA^* = \lambda I$}\}$ be the $n$-dimensional general symplectic group; here $*$ is the symplectic involution
\[
A^* := \Omega^{-1}A^T\Omega
\]
where
\[
\Omega = \begin{pmatrix} 0 & I \\ -I & 0 \end{pmatrix}
\]
is the matrix of the standard symplectic form.  It is a connected reductive algebraic group, and it is in fact an affine subvariety of $GL_{2n}$, hence of $M_{2n} \times \mb{A}^1$.

The results and proofs for $GSp_{2n}$ are exactly analogous to those for $GO_n$, except that instead of starting with the relations $F_{j, n+1}$ defined above, we start with the relations $F^i_{h, n}$, for $1 \le i \le n+1$ and $0 \le h < i$, defined in \cite[Theorem 10.2(a)]{procesi}.  For convenience, we state the analog of Theorem \ref{gon_main_thm}; from this and the original proof, it is easy to read off a finite presentation for $k[GSp_{2n}^\bullet]^{\Ad GSp_{2n}}$ as an $\FFG$-algebra.

Define a function $\lambda: GSp_{2n}(k) \ra k$ by $AA^t = \lambda(A)I$.  Note that $\lambda \in k[GSp_{2n}]^{\Ad GSp_{2n}}$.

\begin{mydef}
Let $\Gamma$ be a group.  A \emph{$GSp_n$-pseudocharacter of $\Gamma$ over $k$} is a pair $(T, l)$, consisting of a set map $T: \Gamma \ra k$ and a group homomorphism $l: \Gamma \ra k^\times$, such that
\begin{itemize}
  \item $T(1) = n$
  \item For all $\gamma_1, \gamma_2 \in \Gamma$, $T(\gamma_1\gamma_2) = T(\gamma_2\gamma_1)$
  \item For all $\gamma \in \Gamma$, $T(\gamma) = l(\gamma)T(\gamma^{-1})$
  \item For all integers $1 \le i \le n+1$ and $0 \le h < i$, and for all $\gamma_1, \dots, \gamma_{n+i} \in \Gamma$, $T$ and $l$ satisfy the relation $H^i_{h, n+1}(l, T, \gamma_1, \dots, \gamma_{n+i}) = 0$, where $H^i_{h, n+1}(l, T, \gamma_1, \dots, \gamma_{n+i})$ is defined as follows:
  \begin{itemize}
    \item Taking $A_1, \dots, A_{n+i}$ to be matrix variables, define $G^i_{h, n+1}(A_1, \dots, A_{n+i})$ to be the same as $F^i_{h, n+1}(A_1, \dots, A_{n+i})$, except that we replace each $\tr(M)$ with formal symbols $l(M')T(M'')$, where $M'$ is the product (with multiplicity) of all letters $A_1, \dots, A_n$ which appear transposed in $M$, and $M'' \in \FG(\{A_1, \dots, A_m\})$ is the result of substituting all transposed letters $A_i^t$ in $M$ with $A_i^{-1}$.
    \item Define $H^i_{h, n+1}(l, T, \gamma_1, \dots, \gamma_{n+i})$ in the same way as $G^i_{h, n+1}(\gamma_1, \dots, \gamma_{n+i})$ by abuse of notation, except that we replace each symbol $l(\gamma)$ with its actual value (using the given $l$), and similarly for each $T(\gamma)$.
  \end{itemize}
\end{itemize}
\end{mydef}

\begin{mydef}
An \emph{$Sp_{2n}$-pseudocharacter of $\Gamma$ over $k$} is a set map $T: \Gamma \ra k$ such that $(T, 1)$ is a $GSp_{2n}$-pseudocharacter.
\end{mydef}

\begin{mythm}\label{gsp_main_thm}
\begin{enumerate}[(i)]
  \item There is a natural bijection between $GSp_{2n}(\overline{k})$-conjugacy classes of semisimple representations $\rho: \Gamma \ra GSp_{2n}(\overline{k})$ and $GSp_{2n}$-pseudocharacters $(T, l)$ of $\Gamma$ over $\overline{k}$.
  The bijection is given by sending $\rho: \Gamma \ra GSp_{2n}(\overline{k})$ to $(\tr(\rho), \lambda(\rho))$, where $\lambda(\rho)(\gamma) := \lambda(\rho(\gamma))$.
  
  \item There is a natural bijection between $Sp_{2n}(\overline{k})$-conjugacy classes of semisimple representations $\rho: \Gamma \ra Sp_{2n}(\overline{k})$ and $Sp_{2n}$-pseudocharacters $T$ of $\Gamma$ over $\overline{k}$.
  The bijection is given by sending $\rho: \Gamma \ra Sp_{2n}(\overline{k})$ to $\tr(\rho)$.
  
  \item If $(T, l)$ (resp.\ $T$) is a $GSp_{2n}$- (resp.\ $Sp_{2n}$-) pseudocharacter over $k$, then the corresponding conjugacy class over $\overline{k}$ given by (i) (resp.\ (ii)) contains a representation $\rho: \Gamma \ra GSp_{2n}(k')$ (resp.\ $\rho: \Gamma \ra Sp_{2n}(k')$) for some finite extension $k'/k$.
  
  \item If $\Gamma$ is profinite and $k$ is a complete extension of $\Q_l$ for some $l$, then (i) and (ii) hold with ``representation'' replaced by ``continuous representation'' and with ``$GSp_{2n}$-pseudocharacters'' replaced by ``continuous $GSp_{2n}$-pseudocharacters''.  Here $(T, l)$ is called continuous if both $T$ and $l$ are continuous.
\end{enumerate}
\end{mythm}

\subsection{Special Orthogonal Group}\label{so_example}
\subsubsection*{Odd Dimension}
When the dimension is $2n+1$ for some $n$, we have $k[GO_{2n+1}^m]^{\Ad SO_{2n+1}} = k[GO_{2n+1}^m]^{\Ad O_{2n+1}}$, since every orthogonal matrix is $\pm 1$ times a special orthogonal matrix.  By the same reasoning as in the proof of Proposition \ref{go_n_gens}, this equals $k[GO_{2n+1}^m]^{\Ad GO_{2n+1}}$ as well.  Then by Lemma \ref{lem_quotient},
\[
k[SO_{2n+1}^m]^{\Ad SO_{2n+1}} = \frac{k[GO_{2n+1}^m]^{\Ad SO_{2n+1}}}{I \cap k[GO_{2n+1}^m]^{\Ad SO_{2n+1}}},
\]
where $I$ is the ideal of $k[GO_{2n+1}^m]$ generated by the relations $\det(A_i) = 1$ for $A_i$ a coordinate matrix.  Thus $k[SO_{2n+1}^\bullet]^{\Ad SO_{2n+1}}$ is generated by $\tr$ as an $\FFG$-algebra, noting that $\lambda = 1$ when restricted to $k[SO_{2n+1}^\bullet]$.  Also, by extending scalars to $\overline{k}$ and using Hilbert's Nullstellensatz, it is easy to see that for any $m$, the ideal of relations between the generators of $k[SO_{2n+1}^m]^{\Ad SO_{2n+1}}$ is the radical of the ideal generated by the $GO_{2n+1}$ relations with $\lambda = 1$ and the relations $\det(A_i) = 1$ (expressed in terms of $\tr(A_i), \dots, \tr(A_i^{2n+1})$).  Hence the relations between $\tr$ for $k[SO_{2n+1}^\bullet]^{\Ad SO_{2n+1}}$ are generated, up to radical, by the relations for $k[GO_{2n+1}^\bullet]^{\Ad GO_{2n+1}}$ with $\lambda = 1$ and the relation $\det = 1$ expressed in terms of $\tr$.

\begin{mydef} An \emph{(odd-dimensional) $SO_{2n+1}$-pseudocharacter of $G$ over $k$} is an $O_{2n+1}$-pseudocharacter $T: G \ra k$ which additionally satisfies the relation $\det(T)(g) = 1$ for all $g \in G$, where $\det(T)(g)$ is a polynomial in the $T(g^i)$ such that $\det(\tr)(B) = \det(B)$ for all matrices $B$.
\end{mydef}
Then the usual result holds by Corollary \ref{lafforgue_cor} and the above discussion.

\subsubsection*{Even Dimension}
When the dimension is $2n$ for some $n$, the invariant theory of $SO_{2n}$ is more complicated.  Aslaksen, Tan, and Zhu \cite[Theorem 3]{aslaksen_tan_zhu} show that for all $m$, $k[M_{2n}^m]^{\Ad SO_{2n}}$ is generated as a $k$-algebra by $\tr$ and the $n$-argument \textit{linearized Pfaffian} $\pl$, defined as the full polarization of the function
\[
\widetilde{\pf}(W) := \pf(W - W^t)
\]
where $\pf$ is the usual Pfaffian; here the inputs to $\tr$ and $\pl$ are again drawn from $\FS(\{A_1, A_1^t, \dots, A_m, A_m^t\})$.  Then as in Proposition \ref{go_n_gens}, $k[GO_{2n}^\bullet]^{\Ad SO_{2n}}$ is generated as an $\FFG$-algebra by $\tr$, $\pl$, and $\lambda$.

A result due to Rogora \cite{rogora} allows us to determine the relations between these generators up to radical, as follows.

\begin{mylemma}
The $\FFG$-ideal of relations between the generators $\tr$, $\pl$, and $\lambda$ for $k[GO_{2n}^\bullet]^{\Ad SO_{2n}}$ is the radical of the $\FFG$-ideal generated by the relations for $k[GO_{2n}^\bullet]^{\Ad GO_{2n}}$ and the relation described in \cite[Theorem 3.2]{rogora}.
\end{mylemma}
\begin{proof}
Let $R$ be a polynomial in terms of the given generators (i.e., in terms of their images under the internal morphisms in the free $\FFG$-algebra) which maps to 0 in $k[GO_{2n}^\bullet]^{\Ad SO_{2n}}$.  Note that conjugating all inputs to $R$ by an element of $O_{2n}(k) \setminus SO_{2n}(k)$ preserves the value of any generator $\tr(M)$ or $\lambda(M)$ while negating the value of any generator $\pl(M_1, \dots, M_{n})$.  Thus conjugating all inputs of any monomial in $R$ sends that monomial to either itself or its negation; we call the monomial ``even'' in the former case and ``odd'' in the latter case.  Let $R_e$ and $R_o$ be the sums of all even and odd monomials in $R$, respectively.  Then $R_e$ and $-R_0$ are mapped to the same image in $k[GO_{2n}^\bullet]^{\Ad SO_{2n}}$.  Then conjugating all of their image's inputs by an element of $O_{2n}(k) \setminus SO_{2n}(k)$, we see that $R_e$ and $R_0$ also map to the same image in $k[GO_{2n}^\bullet]^{\Ad SO_{2n}}$.  Hence $R_e$ and $R_o$ both map to 0, so that they are both in the $\FFG$-ideal of relations.

It now suffices to show that the even and odd relations are in the given $\FFG$-ideal.  If $R_e$ is an even relation, then each of its monomials consists of traces, lambdas, and pairs of linearized Pfaffians.  After replacing each pair of linearized Pfaffians with a polynomial in traces using the relations described in \cite[Theorem 3.2]{rogora}, we get a polynomial in terms of traces and lambdas which is a $GO_{2n}$-invariant.  Hence $R_e$ is in the given $\FFG$-ideal.  Next, if $R_o$ is an odd relation, then $R_o^2$ is an even relation, hence is in the given $\FFG$-ideal.  Then $R_o$ is in the radical of the given $\FFG$-ideal.
\end{proof}

Then as in the odd dimension case, restricting to $k[SO_{2n}^\bullet]^{\Ad SO_{2n}}$, we find that $k[SO_{2n}^\bullet]^{\Ad SO_{2n}}$ is generated as an $\FFG$-algebra by $\tr$ and $\pl$, and the relations between these generators are generated, up to radical, by the relations for $k[GO_{2n}^\bullet]^{\Ad GO_{2n}}$ with $\lambda = 1$, the relation described in \cite[Theorem 3.2]{rogora}, and the relation $\det = 1$ expressed in terms of $\tr$.

\begin{mydef} An \emph{(even-dimensional) $SO_{2n}$-pseudocharacter of $G$ over $k$} is a pair of functions $T: G \ra k$, $P: G^{n} \ra k$, such that
\begin{itemize}
  \item $T$ is an $O_{2n}$-pseudocharacter of $G$ over $k$
  \item For all $g \in G$, $\det(T)(g) = 1$
  \item For all $g_1, \dots, g_{n}, h_1, \dots, h_{n}$, $P(g_1, \dots, g_{n})P(h_1, \dots, h_{n})$ satisfies the relation in \cite[Theorem 3.2]{rogora} with $P$ in place of $Q$ and $T$ in place of $\tr$.
\end{itemize}
\end{mydef}

Then we have the usual result.

\section{Application: Conjugacy vs.\ Element-Conjugacy}\label{conjugacy}
In this section, we use our pseudocharacters to answer questions about conjugacy vs.\ element-conjugacy of group homomorphisms $\Gamma \ra H$ for $H$ a linear algebraic group, following Larsen \cite{larsen_1, larsen_2}.

\begin{mydef}
Fix a linear algebraic group $H$ over a field $k$, and let $\Gamma$ be another group.  Two homomorphisms $\rho_1, \rho_2: \Gamma \ra H(k)$ are called \emph{globally conjugate} if there exists $h \in H(k)$ such that $\rho_1 = h\rho_2h^{-1}$.  They are called \emph{element-conjugate} if for all $\gamma \in \Gamma$, there exists $h_\gamma \in H(k)$ such that $\rho_1(\gamma) = h_\gamma\rho_2(\gamma)h_\gamma^{-1}$.
\end{mydef}

The ``conjugacy vs.\ element-conjugacy'' question for $H(k)$ asks whether or not element-conjugate semisimple homomorphisms $\Gamma \ra H(k)$ are automatically globally conjugate.

\begin{mydef}
A linear algebraic group $H(k)$ is \emph{acceptable} if element-conjugacy implies global conjugacy for all semisimple representations of arbitrary groups $\Gamma$.  We call $H(k)$ \emph{finite-acceptable} if element-conjugacy implies global conjugacy for all finite groups $\Gamma$, and \emph{compact-acceptable} if element-conjugacy implies conjugacy for all continuous semisimple representations of compact groups $\Gamma$.
\end{mydef}

In \cite{larsen_1, larsen_2}, Larsen mostly classifies the complex and compact simple Lie groups as finite-acceptable or finite-unacceptable (which implies unacceptable).  Recent results by Fang, Han, and Sun \cite{fang_han_sun} show that $GL_n(\C)$, $O_n(\C)$, $Sp_{2n}(\C)$, and their real compact forms are in fact compact-acceptable.

In this section, we give a simple sufficient condition for the acceptability of a connected reductive group $H$ over an algebraically closed field $k$, in terms of the $\FFG$-algebra $k[H^\bullet]^{\Ad H}$.  We then use this condition and the results of Section \ref{examples} to immediately show that $GO_n(\C)$, $O_n(\C)$, $GSp_{2n}(\C)$, and $Sp_{2n}(\C)$ are acceptable (not just finite- or compact-acceptable).  By \cite[Proposition 1.7]{larsen_1}, this also implies that the maximal compact subgroups of these groups are compact-acceptable, a result which was proviously known only for $O_n(\R)$ and $Sp_{2n}(\R)$.

We also use our pseudocharacters for $SO_{2n}$ to give a criterion for when a semisimple representation $\rho: \Gamma \ra SO_{2n}(k)$ is a counterexample to acceptability for $SO_{2n}(k)$.  We then construct a counterexample to acceptability for $SO_{2n}(\C)$ ($n \ge 3$) for the domain group $\Gamma = \Z/4\Z \times \Z/4\Z$; this gives a simpler example than the one in \cite[Proposition 3.8]{larsen_1}, and it additionally shows that $SO_6(\C)$ is unacceptable, a result which was not previously known.

\subsection{General Principles}
Let $H$ be a linear algebraic group.  Suppose that $H$ has pseudocharacters consisting of one-argument functions only.  More formally, let $k$ be an algebraically closed field of characteristic 0, and suppose that there exist invariants $f_1, \dots, f_n \in k[H]^{\Ad H}$ such that for any group $\Gamma$, the map $\rho \mapsto (f_1(\rho), \dots, f_n(\rho))$ induces a bijection between
\[
\{\mbox{$H(k)$-conjugacy classes of semisimple representations $\rho: \Gamma \ra H(k)$}\}
\]
and
\[
\{\mbox{maps $F_1, \dots, F_n: \Gamma \ra k$ satisfying certain fixed relations}\}.
\]
Then $H(k)$ is acceptable: indeed, if $\rho_1, \rho_2: \Gamma \ra H(k)$ are semisimple element-conjugate representations, then for all $\gamma \in \Gamma$, we have $f_i(\rho_1(\gamma)) = f_i(\rho_2(\gamma))$ for $1 \le i \le n$, so $\rho_1$ and $\rho_2$ have the same $H$-pseudocharacter, hence they are conjugate.

When $H$ is connected reductive (or, more generally, when the conclusion of Corollary \ref{lafforgue_cor} holds for $H$, such as for $H = O_n$), we can restate this result as follows.

\begin{mythm}
Let $H$ be an algebraic group over an algebraically closed field $k$ of characteristic 0 such that Corollary \ref{lafforgue_cor} holds for $H$ (e.g., $H$ is connected reductive).  Suppose that $k[H^\bullet]^{\Ad H}$ is generated by $k[H]^{\Ad H}$ as an $\FFG$-algebra.  Then $H(k)$ is acceptable.
\end{mythm}

Although it would be convenient if the converse to this theorem were true, it appears to be false.  Upcoming work by Yu \cite[Theorem 4.2(1)]{yu} shows that $SO_4(\R)$ is compact-acceptable, so that $SO_4(\C)$ is finite-acceptable by \cite[Proposition 1.7]{larsen_1}, suggesting that it is also acceptable; meanwhile, from our investigations it appears that the two-argument function $\pl$ in $k[SO_4^\bullet]^{\Ad SO_4}$ is not generated by $k[SO_4]^{\Ad SO_4}$, although we have not proved this definitively.

\subsection{Element-conjugacy vs.\ Conjugacy for \texorpdfstring{$SO_{2n}$}{SO(2n)}}
Let $k$ be an algebraically closed field of characteristic 0, and let $n \ge 2$ be an integer.  We wish to characterize all pairs of semisimple representations $\rho_1, \rho_2: \Gamma \ra SO_{2n}(k)$ which are element-conjugate but not globally conjugate, at least when $\Gamma$ is torsion.  Let $\pl$ denote the linearized antisymmetrized Pfaffian (see Section \ref{so_example} above).  Our result is as follows.

\begin{myprop}
Let $\Gamma$ be a group, and let $\rho: \Gamma \ra SO_{2n}(k)$ be a semisimple representation.  If there exists a semisimple representation $\rho': \Gamma \ra SO_{2n}(k)$ which is element-conjugate but not globally conjugate to $\rho$, then:
\begin{itemize}
  \item For all $\gamma \in \Gamma$, $\det(\rho(\gamma) - \rho(\gamma)^t) = 0$
  \item There exist $\gamma_1, \dots, \gamma_n \in G$ such that $\pl(\rho(\gamma_1), \dots, \rho(\gamma_n)) \neq 0$.
\end{itemize}
If $\Gamma$ is torsion, then the converse holds as well.

When such a $\rho'$ exists, it is unique up to conjugation by $SO_{2n}(k)$, and it is given by
\[
\rho' = X\rho X^{-1}
\]
for any $X \in O_{2n}(k) \setminus SO_{2n}(k)$.
\end{myprop}
\begin{proof}
\textbf{Uniqueness:} Let $\rho'$ be element-conjugate but not globally conjugate to $\rho$ in $SO_{2n}$.  Then $\rho$ and $\rho'$ are element-conjugate in $O_{2n}$, hence globally conjugate in $O_{2n}$.  Thus there is an $X \in O_{2n}(k)$ such that $\rho' = X\rho X^{-1}$, and necessarily $X \notin SO_{2n}(k)$.  Since $SO_{2n}(k)$ has index 2 in $O_{2n}(k)$, any other choice of $X$ gives a representation which is globally conjugate to $\rho'$ in $SO_{2n}(k)$.

\textbf{Existence, $(\Longrightarrow)$:} Let $\rho'$ be a semisimple representation which is element-conjugate but not globally conjugate to $\rho$.  By the uniqueness proof, we can write $\rho' = X\rho X^{-1}$ as above.  Let $\pl$ denote the linearized antisymmetrized Pfaffian, which is an odd $n$-ary invariant in $k[SO_{2n}^\bullet]^{\Ad SO_{2n}}$.  Here ``odd'' means that
\begin{equation}\label{eq_oddness}
\pl(\rho(\gamma_1), \dots, \rho(\gamma_n)) = -\pl(X\rho(\gamma_1)X^{-1}, \dots, X\rho(\gamma_n)X^{-1}) = -\pl(\rho'(\gamma_1), \dots, \rho'(\gamma_n))
\end{equation}
for all $\gamma_1, \dots, \gamma_n \in \Gamma$.  Since $\rho$ and $\rho'$ are not globally conjugate, they must have different pseudocharacters, and since $\tr(\rho) = \tr(\rho')$ by element-conjugacy, there must exist $\gamma_1, \dots, \gamma_n \in \Gamma$ such that
\[
\pl(\rho(\gamma_1), \dots, \rho(\gamma_n)) \neq \pl(\rho'(\gamma_1), \dots, \rho'(\gamma_n)).
\]
Then by (\ref{eq_oddness}), $\pl(\rho(\gamma_1), \dots, \rho(\gamma_n)) \neq 0$.

Next, since $\rho$ and $\rho'$ are element-conjugate, $\rho|_{\langle \gamma \rangle}$ is conjugate to $\rho'|_{\langle \gamma \rangle}$ in $SO_{2n}$ for each $\gamma \in \Gamma$, so
\[
\pl(\rho(\gamma^{m_1}), \dots, \rho(\gamma^{m_n})) = \pl(\rho'(\gamma^{m_1}), \dots, \rho(\gamma^{m_n}))
\]
for all $\gamma \in \Gamma$ and $m_1, \dots, m_n \in \Z$.  Then by (\ref{eq_oddness}), $\pl(\rho(\gamma^{m_1}), \dots, \rho(\gamma^{m_n})) = 0$.  In particular, $\widetilde{\pf}(\rho(\gamma)) = \frac{1}{n!}\pl(\rho(\gamma), \dots, \rho(\gamma)) = 0$ for all $\gamma \in \Gamma$.  Hence
\[
\det(\rho(\gamma) - \rho(\gamma)^t) = \pf(\rho(\gamma) - \rho(\gamma)^t)^2 = \widetilde{\pf}(\rho(\gamma))^2 = 0.
\]

\textbf{Existence, $(\Longleftarrow)$:} Let $X \in O_{2n}(k) \setminus SO_{2n}(k)$, and set $\rho'(\gamma) = X\rho(\gamma)X^{-1}$.  Then by assumption, there exist $\gamma_1, \dots, \gamma_n$ such that
\[
\pl(\rho(\gamma_1), \dots, \rho(\gamma_n)) \neq -\pl(\rho(\gamma_1), \dots, \rho(\gamma_n)) = \pl(\rho'(\gamma_1), \dots, \rho'(\gamma_n)),
\]
so $\rho$ and $\rho'$ are not globally conjugate.

Now fix $\gamma \in \Gamma$.  Since $\Gamma$ is torsion, Maschke's Theorem implies that both $\rho|_{\langle \gamma \rangle}$ and $\rho'|_{\langle \gamma \rangle}$ are semisimple.  Thus to show that $\rho|_{\langle \gamma \rangle}$ and $\rho'|_{\langle \gamma \rangle}$ are conjugate in $SO_{2n}$, it suffices to show that they have the same $SO_{2n}$-pseudocharacters.  They have the same traces because $\rho$ and $\rho'$ are conjugate in $O_{2n}$.  To show that they have the same linearized Pfaffians, we must show
\[
\pl(\rho(\gamma^{m_1}), \dots, \rho(\gamma^{m_n})) = 0
\]
for all $m_1, \dots, m_n \in \Z$, since the corresponding Pfaffian for $\rho'$ is the negative of that for $\rho$.  By definition, $\pl(\rho(\gamma^{m_1}), \dots, \rho(\gamma^{m_n}))$ is the multilinear term in
\[
\widetilde{\pf}\left(t_1\rho(\gamma^{m_1}) + \dots + t_n\rho(\gamma^{m_n})\right) = \pf\left(t_1(\rho(\gamma^{m_1}) - \rho(\gamma^{m_1})^t) + \dots + t_n(\rho(\gamma^{m_n}) - \rho(\gamma^{m_n})^t\right).
\]
But $\rho(\gamma) - \rho(\gamma)^t = \rho(\gamma) - \rho(\gamma)^{-1}$ divides $\rho(\gamma)^{m_i} - \rho(\gamma)^{-m_i} = \rho(\gamma^{m_i}) - \rho(\gamma^{m_i})^t$ for all $i$, so the assumption $\det(\rho(\gamma) - \rho(\gamma)^t) = 0$ implies that
\[
\det\left(t_1(\rho(\gamma^{m_1}) - \rho(\gamma^{m_1})^t) + \dots + t_n(\rho(\gamma^{m_n}) - \rho(\gamma^{m_n})^t)\right) = 0.
\]
Hence taking the square root, the Pfaffian is zero as well for all values of $t_1, \dots, t_n$.  Thus $\pl(\rho(\gamma^{m_1}), \dots, \rho(\gamma^{m_n})) = 0$, proving the claim.
\end{proof}

\subsection{A Finite Abelian Counterexample for \texorpdfstring{$SO_{2n}$, $n \ge 3$}{SO(2n), n >= 3}}
Let $\Gamma = \Z/4\Z \times \Z/4\Z$, with generators $(1, 0)$ and $(0, 1)$.  Let
\[
A = \begin{pmatrix} 0 & 1 \\ -1 & 0 \end{pmatrix} \in SO_2(\C).
\]
Define a homomorphism $\rho_6: \Gamma \ra SO_6(\C)$ by
\begin{align*}
\rho_6(1, 0) &= A \oplus A \oplus I, \\
\rho_6(0, 1) &= I \oplus A \oplus A.
\end{align*}
Then one can check that $\det(\rho_6(\gamma) - \rho_6(\gamma)^t) = 0$ for all $\gamma \in \Gamma$, while $\pl(\rho_6(1, 0), \rho_6(0, 1), \rho_6(0, 1)) = 16$.  Hence $\rho_6$ is a counterexample to element-conjugacy implying conjugacy for $SO_6(\C)$.

More generally, we have:
\begin{myprop}
Let $\Gamma$ and $A$ be as above.  For any $n \ge 3$, the homomorphism $\rho_{2n}: \Gamma \ra SO_{2n}(\C)$ defined by
\begin{align*}
\rho_{2n}(1, 0) &= A \oplus A \oplus I \oplus \bigoplus_{i=4}^n A, \\
\rho_{2n}(0, 1) &= I \oplus A \oplus A \oplus \bigoplus_{i=4}^n A
\end{align*}
satisfies $\det(\rho_{2n}(\gamma) - \rho_{2n}(\gamma)^t) = 0$ for all $\gamma \in \Gamma$ and $\pl(\rho_{2n}(1, 0), \rho_{2n}(0, 1), \dots, \rho_{2n}(0, 1)) \neq 0$.  Hence $\rho_{2n}$ gives a counterexample to element-conjugacy implying global conjugacy.
\end{myprop}
\begin{proof}
Let $\gamma \in \Gamma$, and write $\rho(\gamma) = \bigoplus_{i=1}^n B^{(i)}$.  We have
\begin{align*}
\det\left( \rho(\gamma) - \rho(\gamma)^t \right)
&= \det\left( \bigoplus_{i=1}^n (B^{(i)} - (B^{(i)})^t) \right) \\
&= \prod_{i=1}^n \det(B^{(i)} - (B^{(i)})^t).
\end{align*}
Hence to show $\det(\rho(\gamma) - \rho(\gamma)^t) = 0$, it suffices to prove that some $2\times 2$ diagonal block $B^{(i)}$ of $\rho(\gamma)$ satisfies $\det(B^{(i)} - (B^{(i)})^t) = 0$.  But one can check that for all $\gamma \in \Gamma$, one of the first three $2 \times 2$ diagonal blocks is a symmetric matrix.

Next, recall that for matrices $C_1, \dots, C_n$, $\pl(C_1, \dots, C_n)$ is defined to be the coefficient of $t_1 \cdots t_n$ in $\pf(t_1(C_1 - C_1^t) + \dots + t_n(C_n - C_n^t))$.  Letting each $C_j = \bigoplus_{i=1}^n C_j^{(i)}$ for some $2 \times 2$ matrices $C_j^{(i)}$, we have
\[
\pf(t_1(C_1 - C_1^t) + \dots + t_n(C_n - C_n^t))
= \prod_{i=1}^n \pf(t_1(C_1^{(i)} - (C_1^{(i)})^t) + \dots + t_1(C_n^{(i)} - (C_n^{(i)})^t).
\]
Now $\pf$ is a linear function of $2 \times 2$ antisymmetric matrices, so this equals
\[
\prod_{i=1}^n \sum_{j=1}^n t_j \pf(C_j^{(i)} - (C_j^{(i)})^t).
\]
Taking the coefficient of $t_1 \cdots t_n$ in this formula, we find that
\[
\pl(C_1, \dots, C_n) = \sum_{\sigma \in S_n} \prod_{i=1}^n \pf(C_{\sigma(i)}^{(i)} - (C_{\sigma(i)}^{(i)})^t).
\]

Finally, note that $\pf(A - A^t) = 2$ and $\pf(I - I^t) = 0$.  Thus
\[
\pl\left(D_1 := \rho_{2n}(1, 0), D_2 := \rho_{2n}(0, 1), \dots, D_n := \rho_{2n}(0, 1)\right)
\]
will be positive so long as for some $\sigma \in S_n$, for all $i$, $D_{\sigma(i)}^{(i)} = A$.  Taking $\sigma$ to be the identity permutation works.
\end{proof}

\begin{mycor}
For all $n \ge 3$ and all odd primes $p$ such that $p \equiv 1 \pmod{4}$, there is a continuous semisimple representation $\rho: \Gal(\overline{\Q_p}/\Q_p) \ra SO_{2n}(\C)$ which is a counterexample to acceptability.
\end{mycor}
\begin{proof}
We need merely construct a Galois extension $K$ of $\Q_p$ with $\Gal(K / \Q_p) \cong \Z/4\Z \times \Z/4\Z$.  This is easily done using Kummer theory, since by assumption, $\mu_4 \subset \Q_p$, and $p$ together with a lift of some generator for $\F_p^\times / (\F_p^\times)^4$ generate a subgroup of $\Q_p^\times / (\Q_p^\times)^4$ isomorphic to $\Z/4\Z \times \Z/4\Z$.
\end{proof}

\noindent{\bf Acknowledgments:} I would like to heartily thank Xinwen Zhu for proposing the idea for this project and for mentoring me throughout it.

\bibliography{pseudocharacters_v2_bib}{}
\bibliographystyle{plain}

\end{document}